\documentclass{amsart}
\usepackage[margin=1in]{geometry}

\usepackage{amsmath,amsthm, amssymb,mathrsfs,  xypic}
\usepackage{stackengine,scalerel}
\usepackage[colorlinks]{hyperref}
\usepackage[capitalize]{cleveref}

\newcommand{\Ext}{\mathrm{Ext}}

\newcommand{\GL}{\mathrm{GL}}

\newcommand{\crys}{\mathrm{crys}}

\newcommand{\mr}[1]{\mathrm{#1}}

\theoremstyle{plain}

\newtheorem*{theorem*}{Theorem}
\newtheorem*{conjecture*}{Conjecture}
\newtheorem{theorem}{Theorem}
\newtheorem{corollary}{Corollary}

\newtheorem{lemma}{Lemma}

\theoremstyle{definition}

\newtheorem{remark}{Remark}

\crefformat{section}{\S#2#1#3}
\crefformat{Section}{\S#2#1#3}

\title{The de Jong fundamental group of $\mathbb{P}^1_C$ depends on $C$ and is not always topologically countably generated.}

\author{Sean Howe}

\begin{document}

\begin{abstract}
For $C/\mathbb{Q}_p$ complete and algebraically closed, we show that the de Jong fundamental group $\pi_{1,\mathrm{dJ}}(\mathbb{P}^1_C)$ depends on $C$ and, if $C$ has cardinality  $>2^{\mathbb{N}}$, that it is not topologically countably generated. The result and proofs generalize to any connected rigid analytic variety with a non-constant map to $\mathbb{P}^n_C$. 
\end{abstract}

\maketitle

\newcommand{\rig}{\mr{rig}}
\newcommand{\dJ}{\mr{dJ}}
Let $C/\mathbb{Q}_p$ be a complete algebraically closed extension and consider the projective line $\mathbb{P}^1_C$ as a rigid analytic variety over $C$. For $t \in C$, we have a map $\varphi_t: \mathbb{P}^1_C \rightarrow \mathbb{P}^1_C,\; [x:y] \mapsto [(x-ty)^2:y^2].$ 
On the open subvariety $\mathbb{A}^1_C=\mathbb{P}^1_C\backslash\{[1:0]\}$ with coordinate $z=x/y$, $\varphi_t$ is the map $z\mapsto (z-t)^2$. 
Let $\mathbb{L}$ be the Lubin-Tate rank two $\mathbb{Q}_p$-local system on the rigid analytic variety $\mathbb{P}^{1}_C$ described in \cite[Proposition 7.2]{deJong.EtaleFundamentalGroupsOfNonArchimedeanAnalyticSpaces} and let $\mathbb{L}_t := \varphi_t^* \mathbb{L}$. 

\begin{lemma}\label{lemma.non-iso}
    If $t_1 \neq t_2$, then $\mathbb{L}_{t_1}$ is not isomorphic to $\mathbb{L}_{t_2}$. 
\end{lemma}
\begin{proof}
    Let $\kappa: T_{\mathbb{P}^1_C} \rightarrow \mathcal{E}nd(\mathbb{L} \otimes{\hat{\mathcal{O}}})(-1)$ be the geometric Sen morphism associated to $\mathbb{L} \otimes \hat{\mathcal{O}}$ by \cite[Theorem 1.0.3]{RodriguezCamargo.GeometricSenTheoryOverRigidAnalyticSpaces}. By its  computation in \cite[\S 4.3]{DospinescuRodriguezCamargo.JacquetLanglands} (or the original computation in  \cite{Pan.OnLocallyAnalyticVectorsOfTheCompletedCohomologyOfModularCurves} after restriction to a supersingular disk), $\kappa$ is injective at every geometric point of $\mathbb{P}^1_C$ (cf. \cite[Lemma 2]{Howe.ThedeJongFundamentalOfANonTrivialAbelianVarietyIsNonAbelian}). Now, for $t \in C$, functoriality of the geometric Sen morphism  \cite[Theorem 1.0.3-(5)]{RodriguezCamargo.GeometricSenTheoryOverRigidAnalyticSpaces} implies that the geometric Sen morphism of $\mathbb{L}_t$ is $\varphi_t^* \kappa \circ d\varphi_t$. In particular, writting $\kappa_i$ for the geometric Sen morphism of $\mathbb{L}_{t_i}$, we find $\kappa_{1}$ is zero on the fiber $T_{\mathbb{P}^1_C, [t_1:1]}$ whereas $\kappa_2$ is not (the derivative of $z \mapsto (z-t)^2$ on $\mathbb{A}^1_C$ vanishes exactly at $t$). Thus $\mathbb{L}_1 \not\cong\mathbb{L}_2$. 
\end{proof}

\begin{theorem}\label{theorem.generators}
    Let $\pi_{1,\dJ}(\mathbb{P}^1_C, [0:1])$ be the de Jong fundamental group of \cite{deJong.EtaleFundamentalGroupsOfNonArchimedeanAnalyticSpaces}. If $S \subseteq \pi_{1,\dJ}(\mathbb{P}^1_C, [0:1])$ is infinite and $2^S$ has cardinality less than that of $C$, then $S$ is not a set of topological generators for $\pi_{1,\dJ}(\mathbb{P}^1_C, [0:1])$. In particular, if $C$ has cardinality $>2^{\mathbb{N}}$ then $\pi_{1,\dJ}(\mathbb{P}^1_C, [0:1])$ is not topologically countably generated. 
\end{theorem}
\begin{proof}
    By \cref{lemma.non-iso}, $\{\mathbb{L}_t\}_{t\in C}$ is a collection of non-isomorphic rank two $\mathbb{Q}_p$-local systems indexed by $C$. On the other hand, by \cite[Theorem 4.2]{deJong.EtaleFundamentalGroupsOfNonArchimedeanAnalyticSpaces}, the isomorphism classes of rank two local systems on $\mathbb{P}^1_C$ are the same as isomorphism classes of continuous representations of $\pi_{1,\dJ}(\mathbb{P}^1_C, [0:1])$ on 2-dimensional $\mathbb{Q}_p$-vector spaces. If $S$ is a set of topological generators, then the cardinality of this latter set is bounded by the cardinality of the set $\GL_2(\mathbb{Q}_p)^S$. Since $\GL_2(\mathbb{Q}_p)$ has cardinality $2^{\mathbb{N}}$, this is the same cardinality as $(2^{\mathbb{N}})^S=2^{\mathbb{N} \times S}$. Since we assume $S$ is infinite,  this is the same cardinality as $2^S$, and thus we conclude.  
\end{proof}

\begin{corollary}\label{cor.dep-choice}
     The isomorphism class of $\pi_{1,\mathrm{dJ}}(\mathbb{P}^1_C,[0:1])$ depends on the choice of $C$.
\end{corollary}
\begin{proof}
    For a fixed $C$, take $C'$ of cardinality larger than the power set of $\pi_{1,\mathrm{dJ}}(\mathbb{P}^1_{C},[0:1])$. Then, \cref{theorem.generators} implies $\pi_{1,\mathrm{dJ}}(\mathbb{P}^1_{C'},[0:1]) \not\cong \pi_{1,\mathrm{dJ}}(\mathbb{P}^1_{C},[0:1])$ (since $S=\pi_{1,\mathrm{dJ}}(\mathbb{P}^1_{C},[0:1])$ generates $\pi_{1,\mathrm{dJ}}(\mathbb{P}^1_{C},[0:1])$).
\end{proof}

The role of \cref{lemma.non-iso} in the proof of \cref{theorem.generators} can also be replaced by the following result.
\begin{lemma}\label{lemma.rank-three}
    If $C$ has cardinality greater that $2^\mathbb{N}$, then the set of isomorphism classes of rank three local systems arising as extensions of the constant rank one local system $\mathbb{Q}_p$ by $\mathbb{L}$ has the same cardinality as $C$.
\end{lemma}
\begin{proof}
    By our assumption on the cardinality of $C$, it suffices to show $\Ext^1(\mathbb{Q}_p, \mathbb{L})$ has the same cardinality as $C$ (the extension structures realized in a fixed rank three local system have cardinality at most $2^{\mathbb{N}}$). But, by the main result of \cite{Hansen.ACounterexample},  $\Ext^1(\mathbb{Q}_p, \mathbb{L})=H^1(\mathbb{P}^1_C, \mathbb{L})= C^2/(B^+_\crys)^{\varphi^2=p}$, and $C^2/(B^+_\crys)^{\varphi^2=p}$ has the same cardinality as $C$ (it can also be presented as a quotient of $C$ by a $2$-dimensional $\mathbb{Q}_p$-subspace). 
\end{proof}

\begin{remark}\label{remark.gen} The constructions of \cref{lemma.non-iso} and \cref{lemma.rank-three} generalize to any rigid analytic variety admitting a non-constant map to $\mathbb{P}^1_C$ and, using the higher rank Lubin-Tate local systems, to any rigid analytic variety admitting a non-constant map to $\mathbb{P}^n_C$. \cref{theorem.generators} and \cref{cor.dep-choice} thus also hold in this generality --- in particular, they apply to any quasi-projective rigid analytic variety over $C$. See \cite[\S4]{Howe.ThedeJongFundamentalOfANonTrivialAbelianVarietyIsNonAbelian} for details.\end{remark}

\noindent\emph{Acknowledgements.} We thank David Hansen, Sasha Petrov, and Bogdan Zavyalov for helpful conversations. This work was supported in part by National Science Foundation grants DMS-2201112 and DMS-2501816.

\bibliographystyle{plain}
\bibliography{references, preprints}

\end{document}